\documentclass[11pt, reqno]{amsart}
\usepackage{lineno,hyperref}
\usepackage{cases}
\usepackage[square, comma, sort&compress, numbers]{natbib}%参考文献连续编号
\usepackage{float}
\usepackage{pifont}
\usepackage{bbding}
\usepackage{amssymb}
\usepackage{mathrsfs}
\usepackage{subfigure}
\usepackage{caption}
\usepackage{geometry}
\usepackage{color}
\usepackage{amsmath, amsthm, amscd, amsfonts, amssymb, graphicx, color}
\usepackage{lineno}
\newtheorem{theorem}{Theorem}

\newtheorem{lemma}{Lemma}

\newtheorem{cor}{Corollary}

\newtheorem{rem}{Remark}
\newtheorem{example}{Example}
\newtheorem{problem}{Problem}

\numberwithin{equation}{section}

\newcommand{\abs}[1]{\left\vert#1\right\vert}

\newcommand{\R}{\mbox{$\mathbb{R}$}}

\newcommand{\C}{\mbox{$\mathbb{C}$}}

\newcommand{\D}{\mbox{$\mathbb{D}$}}

\begin{document}
\setcounter{page}{1}

\title[Construction of univalent harmonic mappings convex in one direction]
{Construction of univalent harmonic mappings\\ convex in one direction}

\author[Zhi-Gang Wang, Lei Shi and Yue-Ping Jiang]{Zhi-Gang Wang, Lei Shi and Yue-Ping Jiang}

\vskip.05in
\address{\noindent Zhi-Gang Wang\vskip.03in
School of Mathematics and Computing Science, Hunan
First Normal University, Changsha 410205, Hunan, P. R. China.}

\email{\textcolor[rgb]{0.00,0.00,0.84}{wangmath$@$163.com}}

\address{\noindent Lei Shi\vskip.03in
 School of Mathematics and Statistics, Anyang Normal University,
Anyang 455002, Henan, P. R. China.}

\email{\textcolor[rgb]{0.00,0.00,0.84}{shimath$@$163.com}}

\address{\noindent Yue-Ping Jiang\vskip.03in
School of Mathematics, Hunan University,
Changsha 410082, Hunan, P. R. China.}
\email{\textcolor[rgb]{0.00,0.00,0.84}{ypjiang731$@$163.com}}



\begin{abstract}
 In the present paper, we derive several conditions of linear combinations and convolutions of harmonic
mappings to be univalent and convex in one direction,
one of them gives a partial answer to an open problem proposed by Dorff.
The results presented here provide extensions and improvements of those given in some earlier works. Several examples of univalent harmonic
mappings convex in one direction are also constructed to demonstrate the main results.
\end{abstract}

\vskip.20in

\maketitle

\tableofcontents

\section{Introduction}

In 1984, Clunie and Sheil-Small \cite{CL} had pointed out that many of the classical
results for conformal mappings have analogues for planar harmonic mappings. Since that time, the theory of planar harmonic mappings
  from the perspective of conformal mappings has received much attention, but a number of basic
problems remain unresolved (see \cite{d} and the references therein).

Let $\mathcal{H}$ denote the class of complex-valued harmonic functions $f$ in the open unit disk $\mathbb{D}=\{z\in\C:\abs{z}<1\}$
normalized by $f(0)=f_{\bar{z}}(0)=f_{z}(0)-1=0$. Such functions can be written in the form $f=h+\bar{g}$, where
\begin{equation}\label{11}
h(z)=z+\sum_{n=2}^{\infty}{a_{n}z^{n}}\ {\rm{and}}\ g(z)=\sum_{n=2}^{\infty}{b_{n}z^{n}}
\end{equation}
are analytic in $\mathbb{D}$. A function $f\in\mathcal{H}$
is locally univalent and sense-preserving in $\mathbb{D}$ if and only if
\begin{equation*}\label{12}
\abs{g^{\prime}(z)}<\abs{h^{\prime}(z)}\quad (z\in\D).
\end{equation*}

 We denote by $\mathcal{S}_{\mathcal{H}}^{0}$ the subclass of $\mathcal{H}$ consisting of univalent and sense-preserving harmonic functions.
 Let  $\mathcal{P}$  be the class of functions $p$ of the form
\begin{equation*}\label{13}
p(z)=1+\sum_{n=1}^{\infty}c_{n}z^n\quad(z\in\D),
\end{equation*}
which are analytic in $\D$ and satisfy the condition $\Re\left(p(z)\right)>0$.

 A domain $\Omega\subset\C$
is said to be convex in the direction $\gamma$, if for all $a\in\C$, the set
$\Omega\cap\{a+te^{i\gamma}:\ t\in\R\}$ is either connected or empty. Particularly, a domain is convex in the direction of real ({\it resp.} imaginary) axis if its intersection with each horizontal ({\it resp.} vertical) lines is connected. A function $f\in\mathcal{H}$ is convex in the direction of real ({\it resp.} imaginary) axis if it maps $\D$ onto a domain convex in the direction of real ({\it resp.} imaginary) axis. Clunie and Sheil-Small \cite{CL} introduced the shear construction method to produce a harmonic mapping with a specified dilatation onto a  domain convex in one direction by shearing a given conformal mapping along parallel lines.

For two harmonic functions
\begin{equation*}\label{14}
f=h+\overline{g}=z+\sum_{n=2}^{\infty}a_{n}z^{n}+\sum_{n=1}^{\infty}\overline{b_{n}}\overline{z}^{n},
\end{equation*}
and
\begin{equation*}\label{15}
F=H+\overline{G}=z+\sum_{n=2}^{\infty}A_{n}z^{n}+\sum_{n=1}^{\infty}\overline{B_{n}}\overline{z}^{n},
\end{equation*}
we define the convolution of them by
\begin{equation}\label{16}
f*F=h*H+\overline{g*G}=z+\sum_{n=2}^{\infty}a_{n}A_{n}z^{n}+\sum_{n=1}^{\infty}\overline{b_{n}}\overline{B_{n}}\overline{z}^{n}.
\end{equation}

For some recent investigations on planar harmonic mappings, one can refer to \cite{am,bl,blw,c1,c2,cg,cpw,Dorff,DO,DorffM,d,hm,iko,k,kv,LI,liu,m1,m2,n,s}.

Macgregor \cite{MC} had shown that the convex combination
$tf+(1-t)g\ (0\leqq t\leqq1)$ of analytic functions need not to be univalent, even if $f$ and $g$ are convex functions. Results on linear combinations for analytic case, see (for example) \cite{Douglas,Trimble}. Let $f_{1}=h_{1}+\overline{g_{1}}$ and  $f_{2}=h_{2}+\overline{g_{2}}$ be two harmonic mappings in $\D$, the linear combination $f_{3}$ of $f_{1}$ and $f_{2}$ is given by
\begin{equation*}\label{17}
f_{3}=tf_{1}+(1-t)f_{2}=[th_{1}+(1-t)h_{2}]+[t\overline{g_{1}}+(1-t)\overline{g_{2}}]=h_{3}+\overline{g_{3}}.
\end{equation*}
For this case, Dorff \cite{DO} provided some  sufficient conditions for the linear combination $f_{3}=tf_{1}+(1-t)f_{2}$  to be univalent and convex in the direction of imaginary axis under the assumption $\omega_{1}=\omega_{2}$, where $\omega_{1}$ and $\omega_{2}$ are the dilatation of
$f_{1}$ and $f_{2}$, respectively. He also posed the following open problem.

\begin{problem}\label{p1}
Does the convex combination $f_{3}=tf_{1}+(1-t)f_{2}\ (0\leqq t\leqq1)$ to be univalent and convex in the direction of imaginary axis without the condition $\omega_{1}=\omega_{2}$?
\end{problem}

Furthermore, Wang \textit{et al.} \cite{WA} proved that the linear combination  $f_{3}=tf_{1}+(1-t)f_{2}\ (0\leqq t\leqq1)$ of two harmonic univalent mappings $f_j=h_j+\overline{g_j}\ (j=1,2)$
with $h_{j}+g_{j}=z/(1-z)$  is univalent and convex in the direction of real axis.

Let $\mathcal{A}$ be the subclass of
$\mathcal{H}$ consisting of normalized analytic functions. For $\phi\in\mathcal{A}$ and $\abs{\lambda}=1$, let
\[\mathcal{W}_{\mathcal{H}}^{\lambda}(\phi):=\{h+\overline{g}\in\mathcal{H}:h+\lambda g=\phi\}.\]
We note that the function classes $\mathcal{W}_{\mathcal{H}}^{1}(\phi)$ and $\mathcal{W}_{\mathcal{H}}^{-1}(\phi)$ were introduced by Nagpal and Ravichandran \cite{SU}, which are used to discuss convolution properties of planar harmonic mappings with some special choices of $\phi$.

In this paper, we aim at deriving some conditions for linear combinations and convolutions of harmonic mappings to be univalent and convex in one direction,
one of them gives a partial answer to Problem \ref{p1} proposed by Dorff \cite{DO}. Some examples of univalent harmonic mappings are also constructed to demonstrate the main results.

\vskip.20in
\section{Preliminary results}

In order to derive our main results, we require the following lemmas.

\begin{lemma}\label{Lem7}
Let $\abs{\lambda}=1$ and $\varphi$ is analytic in $\D$. Suppose that $f_{j}=h_{j}+\overline{g_{j}}\in \mathcal{S}_{\mathcal{H}}^{0}$ with
\begin{equation*}
h_{j}-\lambda g_{j}=\varphi\quad(j=1,2)
\end{equation*} for some  $\lambda$ and $\varphi$. Then
$f_{3}=tf_{1}+(1-t)f_{2}\ (0\leqq t\leqq1)$
is locally univalent.
\end{lemma}
\begin{proof} For $f_{3}=tf_{1}+(1-t)f_{2}\ (0\leqq t\leqq1)$, the dilatation of $f_{3}$ is given by
 \begin{equation*}
\omega_{3}=\frac{tg_{1}^{\prime}+(1-t)g_{2}^{\prime}}{th_{1}^{\prime}+(1-t)h_{2}^{\prime}}.
 \end{equation*}
Since $\abs{\lambda\omega_{j}}=\abs{\omega_{j}}<1\ (j=1,2;\abs{\lambda}=1)$ and $h_{j}-\lambda g_{j}=\varphi$, we have
\begin{equation*}\begin{split}
\Re\left(\frac{1+\lambda\omega_{3}}{1-\lambda\omega_{3}}\right)=&\Re\left(\frac{th_{1}^{\prime}+(1-t)h_{2}^{\prime}
+\lambda[tg_{1}^{\prime}+(1-t)g_{2}^{\prime}]}{th_{1}^{\prime}+(1-t)h_{2}^{\prime}-\lambda[tg_{1}^{\prime}+(1-t)g_{2}^{\prime}]}\right)\\
=&t\Re\left(\frac{h_{1}^{\prime}+\lambda g_{1}^{\prime}}{\varphi^{\prime}}\right)
+(1-t)\Re\left(\frac{h_{2}^{\prime}+\lambda g_{2}^{\prime}}{\varphi^{\prime}}\right)\\
=&t\Re\left(\frac{h_{1}^{\prime}+\lambda g_{1}^{\prime}}{h_{1}^{\prime}-\lambda g_{1}^{\prime}}\right)
+(1-t)\Re\left(\frac{h_{2}^{\prime}+\lambda g_{2}^{\prime}}{h_{2}^{\prime}-\lambda g_{2}^{\prime}}\right)\\
=&t\Re\left(\frac{1+\lambda\omega_{1}}{1-\lambda\omega_{1}}\right)
+(1-t)\Re\left(\frac{1+\lambda\omega_{2}}{1-\lambda\omega_{2}}\right)\\
>&0,\end{split}
\end{equation*}
which implies  that $\abs{\omega_{3}}=\abs{\lambda\omega_{3}}<1$, so $f_{3}$ is locally univalent. This completes the proof of Lemma \ref{Lem7}.
\end{proof}

\begin{lemma}\label{Lem1}{\rm{(See \cite{CL})}}
A sense-preserving harmonic function $f=h+\overline{g}$ in $\D$ is a univalent mapping of $\D$ convex in the direction of real $(resp.\ imaginary)$ axis if and only if $h-g$ $(resp.\ h+g)$ is a conformal univalent mapping of $\D$ convex in the direction of real $(resp.\ imaginary)$ axis.
\end{lemma}

\begin{lemma}\label{Lem03}{\rm{(See \cite{CL})}}
A harmonic function $f = h + \overline{g}$ locally univalent in $\D$ is a univalent mapping
of $\D$ onto a domain convex in the direction $\gamma$
 if and only if $h-e^{2i\gamma}
g$ is an analytic univalent
mapping of $\D$ onto a domain convex in the direction $\gamma$.
\end{lemma}

We note that Lemma \ref{Lem03} is a generalization of Lemma \ref{Lem1}.

\begin{lemma}\label{Lem2}{\rm{(See \cite{RO})}}
Let $\varphi$ be a non-constant regular function in $\D$. Then the function $\varphi$ maps $\D$ univalently onto a domain convex in the direction of imaginary axis if and only if there exists two real numbers $\mu$ and $\nu$, where $0\leqq\mu<2\pi$ and $0\leqq\nu\leqq\pi$, such that
\begin{equation}\label{21}
\Re\left(-ie^{i\mu}\left(1-2ze^{-i\mu}\cos\nu+z^{2}e^{-2i\mu}\right)\varphi^{\prime}(z)\right)\geqq0\quad(z\in\D).
\end{equation}
\end{lemma}

\begin{lemma}\label{Lem3}
If there exists an analytic function $p\in\mathcal{P}$ and two constants $\mu, \nu\in[0,\pi]$  such that
\begin{equation}
\phi(z)=\int_{0}^{z}\frac{\cos\mu+i (\sin\mu) p(\zeta)}{e^{i\mu}(1-2\zeta e^{-i\mu}\cos\nu+\zeta^{2}e^{-2i\mu})}d\zeta\quad(z\in\D),
\end{equation}
then $\phi$
 is  univalent and convex in the direction of  imaginary axis.
\end{lemma}
\begin{proof}
By  Lemma \ref{Lem2}, we only need to show that $\phi$ satisfies the condition
\begin{equation*}
\Re\left(-ie^{i\mu}(1-2ze^{-i\mu}\cos\nu+z^{2}e^{-2i\mu})\phi^{\prime}\right)\geqq0.
\end{equation*}
By observing that
\begin{equation*}
\Re\left(-ie^{i\mu}(1-2ze^{-i\mu}\cos\nu+z^{2}e^{-2i\mu})\phi^{\prime}\right)=\Re((\sin\mu)p(z))\geqq0\quad(z\in\D),
\end{equation*}
we obtain the desired conclusion of Lemma \ref{Lem3}.
\end{proof}

The proof of Lemma \ref{Lem4} is similar to that of Lemma \ref{Lem3}, we choose to omit the details.

\begin{lemma}\label{Lem4}
If there exists an analytic function $p\in\mathcal{P}$ and two constants $\mu\in\left[0,\frac{\pi}{2}\right]\cup\left[\frac{3\pi}{2},\pi\right]$, $\nu\in[0,\pi]$  such that
\begin{equation}
\phi(z)=\int_{0}^{z}\frac{\cos\mu p(\zeta)+i \sin\mu}{e^{i\mu}(1-2\zeta e^{-i\mu}\cos\nu+\zeta^{2}e^{-2i\mu})}d\zeta\quad(z\in\D),
\end{equation}
then $\phi$
 is  univalent and convex in the direction of  real axis.
\end{lemma}

\begin{lemma}\label{Lem5}{\rm{(See \cite{CP})}}
Let $f$ be an analytic function  in $\D$ with $f(0)=0$ and $f^{\prime}(0)\neq0$. Suppose also that
\begin{equation}\label{24}
\kappa(z)=\frac{z}{(1+ze^{i\theta_{1}})(1+e^{i\theta_{2}})}\quad(\theta_{1},\theta_{2}\in\R).
\end{equation}
If
\begin{equation}
\Re\left(\frac{zf^{\prime}(z)}{\kappa(z)}\right)>0\quad(z\in\D).
\end{equation}
Then $f$ is convex in the direction of real axis.
\end{lemma}

\begin{lemma}\label{Lem6}{\rm{(See \cite{RU})}}
Let $\xi$ and $\psi$ be convex and starlike functions, respectively, such
that $\xi(0)=\psi(0)=0$. Then for each $F$ analytic in $\D$ and satisfying $\Re\left(F(z)\right)>0$, we have
\begin{equation}
\Re\left(\frac{\psi(z)F(z)*\xi(z)}{\psi(z)*\xi(z)}\right)>0\quad(z\in\D).
\end{equation}
\end{lemma}

\vskip.20in
\section{Main results}

We begin by deriving the following result.

\begin{theorem}\label{t1}
Let $f_{j}=h_{j}+\overline{g}_{j}\in \mathcal{S}_{\mathcal{H}}^{0}\ (j=1,2)$ with  $f_{j}\in \mathcal{W}_{\mathcal{H}}^{1}(\phi)$.
If $\phi$ is convex in the direction of imaginary axis, then $f_{3}=tf_{1}+(1-t)f_{2}\ (0\leqq t\leqq1)$
 is  univalent and convex in the direction of  imaginary axis.
\end{theorem}
\begin{proof}
By taking $\lambda=-1$ in Lemma \ref{Lem7}, we know that $f_{3}$ is locally univalent. Since
$h_{j}+g_{j}=\phi\ (j=1,2)$, we have
\begin{equation*}
h_{3}+g_{3}=\left[th_{1}+(1-t)h_{2}\right]+\left[tg_{1}+(1-t)g_{2}\right]=\phi,
\end{equation*}
which is convex in the direction of imaginary axis by the assumption.
Thus, by Lemma \ref{Lem1}, we know that $f_{3}$ is univalent and convex in the direction of imaginary axis.
\end{proof}

\begin{rem}
{\rm  Theorem \ref{t1} gives a partial answer to Problem \ref{p1}.}
\end{rem}

In view of Theorem \ref{t1} and Lemma \ref{Lem3}, we obtain the following result with  special choice of $\phi$.
\begin{cor}\label{c1}
Suppose that $\gamma\in[0,1]$, $\alpha\in[-1,1]$ and $\theta\in(0,\pi)$. Let $f_{j}=h_{j}+\overline{g_{j}}\in \mathcal{S}_{\mathcal{H}}^{0}$ with
\begin{equation*}
f_{j}\in \mathcal{W}_{\mathcal{H}}^{1}\left(\gamma\frac{z(1-\alpha z)}{1-z^2}+(1-\gamma)\frac{1}{2i\sin\theta}\log\left(\frac{1+ze^{i\theta}}{1+ze^{-i\theta}}\right)\right)\quad(j=1,2).
\end{equation*}
Then
$f_{3}=tf_{1}+(1-t)f_{2}\ (0\leqq t\leqq1)$
 is  univalent and convex in the direction of  imaginary axis.
\end{cor}
\begin{proof}
Define the function $\widetilde{p}(z)$ by
\begin{equation*}
\widetilde{p}(z)=\gamma\frac{1-2\alpha z+z^2}{1-z^2}+(1-\gamma)\frac{1-z^2}{(1+ze^{i\theta})(1+ze^{-i\theta})}\quad(z\in\D),
\end{equation*}
where $0\leqq\gamma\leqq1$, $-1\leqq\alpha\leqq1$ and $0<\theta<\pi$. We know that $\widetilde{p}(z)$ is analytic in $\D$,   $\widetilde{p}(0)=1$, and
\begin{equation*}\begin{split}
\Re\left(\widetilde{p}(z)\right)&=\frac{1}{2}\left(\widetilde{p}(z)+\overline{\widetilde{p}(z)}\right)\\
&=\gamma\frac{(1-\abs{z}^2)\left(1+\abs{z}^2-2\alpha\Re(z)\right)}{{\abs{1-z^2}}^2}+(1-\gamma)\frac{(1-\abs{z}^2)\left(1+\abs{z}^2+2\cos\theta\Re(z)\right)}
{{\abs{(1+ze^{i\theta})(1+ze^{-i\theta})}}^2}
>0.
\end{split}
\end{equation*}
By taking $\mu=\nu=\frac{\pi}{2}$ and $p=\widetilde{p}(z)$ in Lemma \ref{Lem3}, we find that
\begin{equation*}
\phi=\gamma\frac{z(1-\alpha z)}{1-z^2}+(1-\gamma)\frac{1}{2i\sin\theta}\log\left(\frac{1+ze^{i\theta}}{1+ze^{-i\theta}}\right)
 \end{equation*}
 is convex in the direction of imaginary axis. Therefore, by Theorem \ref{t1}, we know that $f_{3}$ is univalent and convex in the direction of imaginary axis.
\end{proof}

The next theorem deals with the linear combination of $f_{1}$ and $f_{2}$
with $f_{1},\ f_2\in \mathcal{W}_{\mathcal{H}}^{-1}(\phi)$.

\begin{theorem}\label{t2}
Let $f_{j}=h_{j}+\overline{g_{j}}\in \mathcal{S}_{\mathcal{H}}^{0}\ (j=1,2)$ with
 $f_{j}\in \mathcal{W}_{\mathcal{H}}^{-1}(\phi)$. If $\phi$ is convex in the direction of real axis,
then
$f_{3}=tf_{1}+(1-t)f_{2}\ (0\leqq t\leqq1)$
is  univalent and convex in the direction of  real axis.
\end{theorem}
\begin{proof}
By setting $\lambda=1$ in Lemma \ref{Lem7}, we know that $f_{3}$ is locally univalent. Since $h_{j}-g_{j}=\phi\ (j=1,2)$,
we find that
\begin{equation*}
h_{3}-g_{3}=\left[th_{1}+(1-t)h_{2}\right]-\left[tg_{1}+(1-t)g_{2}\right]=\phi,
\end{equation*}
which is convex in the direction of real axis by the assumption. Thus, by Lemma \ref{Lem1}, we know that $f_{3}$ is univalent and convex in the direction of real axis.
\end{proof}

Combining Theorem \ref{t2} and Lemma \ref{Lem4}, we obtain the following corollary.
\begin{cor}\label{c2}
Suppose that $\gamma\in[0,1]$ and $\beta\in[-2,2]$. Let $f_{j}=h_{j}+\overline{g_{j}}\in \mathcal{S}_{\mathcal{H}}^{0}$ with
\begin{equation*}
f_{j}\in \mathcal{W}_{\mathcal{H}}^{-1}\left(\gamma\left(\frac{1}{2}\log\frac{1+z}{1-z}\right)+(1-\gamma)\frac{z}{1+\beta z+z^2}\right)\quad(j=1,2).
\end{equation*}
Then
$f_{3}=tf_{1}+(1-t)f_{2}\ (0\leqq t\leqq1)$
 is  univalent and convex in the direction of  real axis.
\end{cor}
\begin{proof}
Let
\begin{equation*}
\widehat{p}(z)=\gamma \frac{1+\beta z+z^2}{1-z^2}+(1-\gamma)\frac{1-z^2}{1+\beta z+z^2}\quad(z\in\D;\ 0\leqq\gamma\leqq1;\ -2\leqq\beta\leqq2),
\end{equation*}
we know that $\widehat{p}(z)$ is analytic in $\D$ with $\widehat{p}(0)=1$. If we define
\begin{equation*}
q(z)=\frac{1+\beta z+z^2}{1-z^2}\quad(z\in\D),
\end{equation*}
then
\begin{equation*}
\Re\left(q(z)\right)=\frac{(1-\abs{z}^2)\left(1+\abs{z}^2+\beta\Re(z)\right)}{\abs{1-z^2}^2}>0\quad(z\in\D).
\end{equation*}
Thus,
\begin{equation*}
\Re\left(\widehat{p}(z)\right)=\gamma\Re\left(q(z)\right)+(1-\gamma)\frac{\Re\left(q(z)\right)}{\abs{q(z)}^2}>0\quad(z\in\D).
\end{equation*}
By taking $\mu=0$, $\cos\nu=-\frac{\beta}{2}$ and $p=\widehat{p}(z)$ in Lemma \ref{Lem4}, we get
\begin{equation*}
\phi=\gamma\left(\frac{1}{2}\log\frac{1+z}{1-z}\right)+(1-\gamma)\frac{z}{1+\beta z+z^2}\quad(0\leqq\gamma\leqq1)
\end{equation*}
is convex in the direction of real axis. Thus, by Theorem \ref{t2}, we know that $f_{3}$ is univalent and convex in the direction of real axis.
\end{proof}

By virtue of Theorems \ref{t1}, \ref{t2} and Lemma \ref{Lem03}, we get the following generalization of Theorems \ref{t1}, \ref{t2}, here we choose to omit the details
of proof.

\begin{theorem}\label{t03}
Let $f_{j}=h_{j}+\overline{g_{j}}\in \mathcal{S}_{\mathcal{H}}^{0}\ (j=1,2)$ with
 $f_{j}\in \mathcal{W}_{\mathcal{H}}^{\lambda}(\phi)$. If $\phi$ is convex in the direction $-\frac{1}{2}\arg\lambda$,
then
$f_{3}=tf_{1}+(1-t)f_{2}\ (0\leqq t\leqq1)$
is  univalent and convex in the direction $-\frac{1}{2}\arg\lambda$.
\end{theorem}

\begin{rem}
{\rm By putting $\lambda=1$ and $\phi=z/(1-z)$ in Theorem \ref{t03}, we get the corresponding result obtained in \cite[Theorem 3]{WA}.}
\end{rem}

For the linear combination  $f_{3}=tf_{1}+(1-t)f_{2}$ with $f_{1}\in \mathcal{W}_{\mathcal{H}}^{1}(\phi)$ and $f_{2}\in \mathcal{W}_{\mathcal{H}}^{1}(\psi)$, some additional conditions are posed to guarantee the local univalency of $f_{3}$.

\begin{theorem}\label{t3}
Let $f_{j}=h_{j}+\overline{g_{j}}\in \mathcal{S}_{\mathcal{H}}^{0}\ (j=1,2)$ with
 $f_{1}\in \mathcal{W}_{\mathcal{H}}^{1}(\phi)$ and $f_{2}\in \mathcal{W}_{\mathcal{H}}^{1}(\psi)$. Suppose also that $\Re\left((1-\omega_{1}\overline{\omega}_{2})h_{1}^{\prime}\overline{h_{2}^{\prime}}\right)\geqq0$
and $t\phi+(1-t)\psi$ is convex in the direction of imaginary axis. Then
$f_{3}=tf_{1}+(1-t)f_{2}\ (0\leqq t\leqq1)$
 is  univalent and convex in the direction of imaginary axis.
\end{theorem}
\begin{proof}
Since $\Re\left((1-\omega_{1}\overline{\omega}_{2})h_{1}^{\prime}\overline{h_{2}^{\prime}}\right)\geqq0$,  from \cite[Theorem 2]{WA}, we know that $f_{3}$ is locally univalent. In view of
\begin{equation*}
h_{3}+g_{3}=\left[th_{1}+(1-t)h_{2}\right]+\left[tg_{1}+(1-t)g_{2}\right]=t\phi+(1-t)\psi
\end{equation*}
is convex in the direction of imaginary axis,  by Lemma \ref{Lem1}, we know that $f_{3}$ is univalent and convex in the direction of imaginary axis.
\end{proof}

For $f_{1}\in \mathcal{W}_{\mathcal{H}}^{-1}(\phi)$ and $f_{2}\in \mathcal{W}_{\mathcal{H}}^{-1}(\psi)$, by applying
the similar method of proof in Theorem \ref{t3}, we get the following result.
\begin{cor}\label{t4}
Let $f_{j}=h_{j}+\overline{g_{j}}\in \mathcal{S}_{\mathcal{H}}^{0}\ (j=1,2)$ with
 $f_{1}\in \mathcal{W}_{\mathcal{H}}^{-1}(\phi)$ and $f_{2}\in \mathcal{W}_{\mathcal{H}}^{-1}(\psi)$. Suppose that $\Re\left((1-\omega_{1}\overline{\omega}_{2})h_{1}^{\prime}\overline{h_{2}^{\prime}}\right)\geqq0$ and $t\phi+(1-t)\psi$ is convex in the direction of real axis.
Then
$f_{3}=tf_{1}+(1-t)f_{2}\ (0\leqq t\leqq1)$
 is  univalent and convex in the direction of real axis.
\end{cor}

Now, we consider the linear combination of $f_{1}$
and $f_{2}$ with $f_{1}\in \mathcal{W}_{\mathcal{H}}^{1}(\phi)$ and  $f_{2}\in \mathcal{W}_{\mathcal{H}}^{-1}(\phi)$.

\begin{theorem}\label{t5}
Let $f_{j}=h_{j}+\overline{g_{j}}\in \mathcal{S}_{\mathcal{H}}^{0}\ (j=1,2)$ with
 $f_{1}\in \mathcal{W}_{\mathcal{H}}^{1}(\phi)$ and $f_{2}\in \mathcal{W}_{\mathcal{H}}^{-1}(\phi)$. Suppose also that $\Re\left(1+\frac{\omega_{1}-\overline{\omega}_{2}}{1-\omega_{1}\overline{\omega}_{2}}\right)\geqq0$ and there are two constants $\theta_{1},\theta_{2}\in\R$ such that
\begin{equation*}
\phi(z)=\int_{0}^{z}\frac{1}{(1+\zeta e^{i\theta_{1}})(1+\zeta e^{i\theta_{2}})}d\zeta\quad(z\in\D).
\end{equation*}
Then
$f_{3}=tf_{1}+(1-t)f_{2}\ (0\leqq t\leqq1)$
is  univalent and convex in the direction of  real axis.
\end{theorem}

\begin{proof}
By applying the similar method as \cite [Theorem 2]{WA}, we know that $f_{3}$ is locally univalent.
Next, we prove that $f_{3}$ is convex in the direction of real axis.
Note that
\begin{equation*}
h_{1}^{\prime}-g_{1}^{\prime}=(h_{1}^{\prime}+g_{1}^{\prime})\left(\frac{h_{1}^{\prime}-g_{1}^{\prime}}{h_{1}^{\prime}+g_{1}^{\prime}}\right)
=\phi^{\prime}(z)\left(\frac{1-\omega_{1}}{1+\omega_{1}}\right)=\phi'(z)p_{1}(z),
\end{equation*}
where \begin{equation}\label{316}p_{1}(z)=\frac{1-\omega_{1}}{1+\omega_{1}}\end{equation} satisfies the condition $\Re\left(p_{1}(z)\right)>0\ (z\in\D)$. Thus, from
 \eqref{24}, we find that
\begin{equation*}\label{034}\begin{split}
\Re\left(\frac{z(h_{3}^{\prime}-g_{3}^{\prime})}{\kappa(z)}\right)&=\Re\left(\frac{z}{\kappa(z)}[t(h_{1}^{\prime}-g_{1}^{\prime})+(1-t)(h_{2}^{\prime}-g_{2}^{\prime})]\right)\\
&=t\Re\left(\frac{z\phi^{\prime}(z)}{\kappa(z)}p_{1}(z)\right)+(1-t)\Re\left(\frac{z\phi^{\prime}(z)}{\kappa(z)}\right)\\
&=t\Re\left(p_{1}(z)\right)+(1-t)>0.\end{split}
\end{equation*}
Therefore, by Lemma \ref{Lem5}, we know that $h_{3}-g_{3}$ is convex in the direction of real axis. Moreover, by Lemma \ref{Lem1}, we deduce that $f_{3}$ is univalent and convex in the direction of real axis.
\end{proof}

In what follows, we provide some results involving the convolution $f_{1}*f_{2}$ to be univalent and convex in one direction.

\begin{theorem}\label{t6}
Let $f_{j}=h_{j}+\overline{g_{j}}\in\mathcal{S}_{\mathcal{H}}^{0}\ (j=1,2)$ with
 $f_{1}\in \mathcal{W}_{\mathcal{H}}^{1}\left(\frac{z}{1-z}\right)$ and $f_{2}\in \mathcal{W}_{\mathcal{H}}^{1}(\phi)$. If there exists two constants $\theta_{1},\theta_{2}\in\R$ such that
\begin{equation*}
\phi(z)=\int_{0}^{z}\frac{1}{(1+\zeta e^{i\theta_{1}})(1+\zeta e^{i\theta_{2}})}d\zeta\quad(z\in\D),
\end{equation*}
then $f_{1}*f_{2}$ is convex in the direction of real axis provided it is locally univalent.
\end{theorem}
\begin{proof}Let
\begin{equation*}
F_{1}=(h_{1}+g_{1})*(h_{2}-g_{2})=h_{1}*h_{2}-h_{1}*g_{2}+g_{1}*h_{2}-g_{1}*g_{2},
\end{equation*}
and
\begin{equation*}
F_{2}=(h_{1}-g_{1})*(h_{2}+g_{2})=h_{1}*h_{2}+h_{1}*g_{2}-g_{1}*h_{2}-g_{1}*g_{2}.
\end{equation*}
Then
\begin{equation*}
\frac{1}{2}\left(F_{1}+F_{2}\right)=h_{1}*h_{2}-g_{1}*g_{2}=H-G,
\end{equation*}
where $H=h_{1}*h_{2}$ and $G=g_{1}*g_{2}$. Clearly, $f_{1}*f_{2}=H+\overline{G}$.

Note that
\begin{equation*}\label{F1}\begin{split}
zF_{1}^{\prime}(z)=&(h_{1}+g_{1})*z(h_{2}-g_{2})^{\prime}
=\frac{z}{1-z}*z(h_{2}-g_{2})^{\prime}\\
=&z(h_{2}^{\prime}-g_{2}^{\prime})
=z\left(\frac{h_{2}^{\prime}-g_{2}^{\prime}}{h_{2}^{\prime}+g_{2}^{\prime}}\right)(h_{2}^{\prime}+g_{2}^{\prime})\\
=&z\left(\frac{1-\omega_{2}}{1+\omega_{2}}\right)(h_{2}^{\prime}+g_{2}^{\prime})
=z\phi^{\prime}(z) p_{2}(z),
\end{split}
\end{equation*}
where \begin{equation}\label{323} p_{2}(z)=\frac{1-\omega_{2}}{1+\omega_{2}}\end{equation}
satisfies the condition $\Re\left(p_{2}(z)\right)>0\ (z\in\D)$.
Similarly, we have
\begin{equation*}\label{F2}\begin{split}
zF_{2}^{\prime}(z)&=z(h_{1}-g_{1})^{\prime}*(h_{2}+g_{2})
=\frac{z}{(1-z)^{2}}\left(\frac{h_{1}^{\prime}-g_{1}^{\prime}}{h_{1}^{\prime}+g_{1}^{\prime}}\right)*\phi(z)\\
&=\frac{z}{(1-z)^{2}}\left(\frac{1-\omega_{1}}{1+\omega_{1}}\right)*\phi(z)
=\frac{z}{(1-z)^{2}}p_{1}(z)*\phi(z),
\end{split}
\end{equation*}
where $p_{1}(z)$ is given by \eqref{316}.
Thus,
\begin{equation*}
z\left(F_{1}^{\prime}(z)+F_{2}^{\prime}(z)\right)=z\phi^{\prime}(z)p_{2}(z)+\frac{z}{(1-z)^{2}}p_{1}(z)*\phi(z).
\end{equation*}
In view of $z\phi^{\prime}(z)=\kappa(z)$, \eqref{24} and Lemma \ref{Lem6}, we deduce that
\begin{equation*}\begin{split}\label{37}
\Re\left(\frac{zF_{1}^{\prime}(z)+zF_{2}^{\prime}(z)}{2\kappa(z)}\right)&=\frac{1}{2}\Re\left(p_{2}(z)\right)+\frac{1}{2}
\Re\left(\frac{\frac{z}{(1-z)^{2}}p_{1}(z)*\phi(z)}{z\phi^{\prime}(z)}\right)\\
&=\frac{1}{2}\Re\left(p_{2}(z)\right)+\frac{1}{2}\Re\left(\frac{\frac{z}{(1-z)^{2}}p_{1}(z)*\phi(z)}{\frac{z}{(1-z)^{2}}*\phi(z)}\right)>0.
\end{split}
\end{equation*}
By Lemma \ref{Lem5}, we know that $H-G$ is convex in the direction of real axis. Furthermore, by Lemma \ref{Lem1}, we conclude that $f_{1}*f_{2}$  is univalent and convex in the direction of real axis provided it is locally univalent.
\end{proof}

\begin{rem}
{\rm Theorem \ref{t6} provides an extension of Theorem 5 in \cite{Dorff}.}
\end{rem}

From Theorem \ref{t6},  we easily get the following result.

\begin{cor}\label{t7}
Let $f_{j}=h_{j}+\overline{g_{j}}\in\mathcal{S}_{\mathcal{H}}^{0}\ (j=1,2)$ with  $f_{1}\in \mathcal{W}_{\mathcal{H}}^{-1}\left(\frac{z}{1-z}\right)$ and $f_{2}\in \mathcal{W}_{\mathcal{H}}^{-1}(\phi)$. If there exists two constants $\theta_{1},\theta_{2}\in\R$ such that
\begin{equation*}\label{38}
\phi(z)=\int_{0}^{z}\frac{1}{(1+\zeta e^{i\theta_{1}})(1+\zeta e^{i\theta_{2}})}d\zeta\quad(z\in\D),
\end{equation*}
then $f_{1}*f_{2}$ is convex in the direction of real axis provided it is locally univalent.
\end{cor}

Finally, we consider the convolution $f_{1}*f_{2}$ with $f_{1}\in \mathcal{W}_{\mathcal{H}}^{-1}\left(\frac{z}{1-z}\right)$ and $f_{2}\in \mathcal{W}_{\mathcal{H}}^{1}(\phi)$.
\begin{theorem}\label{t8}
Let $f_{j}=h_{j}+\overline{g_{j}}\in\mathcal{S}_{\mathcal{H}}^{0}\ (j=1,2)$ with
 $f_{1}\in \mathcal{W}_{\mathcal{H}}^{-1}\left(\frac{z}{1-z}\right)$ and $f_{2}\in \mathcal{W}_{\mathcal{H}}^{1}\left(\frac{1}{2}\log\frac{1+z}{1-z}\right)$.
Then $f_{1}*f_{2}$ is convex in the direction of imaginary axis provided it is locally univalent.
\end{theorem}
\begin{proof}Let
\begin{equation*}
F_{1}=(h_{1}-g_{1})*(h_{2}-g_{2})=h_{1}*h_{2}-h_{1}*g_{2}-h_{2}*g_{1}+g_{1}*g_{2},
\end{equation*}
and
\begin{equation*}
F_{2}=(h_{1}+g_{1})*(h_{2}+g_{2})=h_{1}*h_{2}+h_{1}*g_{2}+h_{2}*g_{1}+g_{1}*g_{2},
\end{equation*}
we obtain
\begin{equation*}
\frac{1}{2}\left(F_{1}+F_{2}\right)=h_{1}*h_{2}+g_{1}*g_{2}=H+G,
\end{equation*}
where $H=h_{1}*h_{2}$ and $G=g_{1}*g_{2}$.
If we set
\begin{equation*}
\Psi(z)=\frac{1}{2}\log\frac{1+z}{1-z},
\end{equation*}
then
\begin{equation*}\label{F3}\begin{split}
zF_{1}^{\prime}(z)&=(h_{1}-g_{1})*z(h_{2}-g_{2})^{\prime}
=\frac{z}{1-z}*z(h_{2}-g_{2})^{\prime}\\
&=z(h_{2}^{\prime}-g_{2}^{\prime})
=z\left(\frac{h_{2}^{\prime}-g_{2}^{\prime}}{h_{2}^{\prime}+g_{2}^{\prime}}\right)\Psi^{\prime}(z)\\
&=z\left(\frac{1-\omega_{2}}{1+\omega_{2}}\right)\Psi^{\prime}(z)
=\frac{z}{1-z^{2}}p_{2}(z),
\end{split}
\end{equation*}
where
$p_{2}(z)$ is given by \eqref{323} and
satisfied the condition  $\Re \left(p_{2}(z)\right)>0\ (z\in\D)$.
Similarly, we have
\begin{equation*}\label{F4}\begin{split}
zF_{2}^{\prime}(z)&=z(h_{1}+g_{1})^{\prime}*(h_{2}+g_{2})
=z(h_{1}^{\prime}-g_{1}^{\prime})\left(\frac{h_{1}^{\prime}+g_{1}^{\prime}}{h_{1}^{\prime}-g_{1}^{\prime}}\right)*\Psi(z)\\
&=\frac{z}{(1-z)^{2}}\left(\frac{1+\omega_{1}}{1-\omega_{1}}\right)*\Psi(z)
=\frac{z}{(1-z)^{2}}p_{3}(z)*\Psi(z),
\end{split}
\end{equation*}
where \begin{equation}p_{3}(z)=\frac{1+\omega_{1}}{1-\omega_{1}}\end{equation} satisfies the condition
$\Re \left(p_{3}(z)\right)>0\ (z\in\D)$. Thus, we obtain
\[z\left(F_{1}^{\prime}(z)+F_{2}^{\prime}(z)\right)=\frac{z}{1-z^{2}}p_{2}(z)+\frac{z}{(1-z)^{2}}p_{3}(z)*\Psi(z).\]
By virtue of $z\Psi^{\prime}(z)=\frac{z}{1-z^{2}}$ and Lemma \ref{Lem6}, we have
\begin{equation*}\label{39}\begin{split}
\Re\left((1-z^{2})\frac{F_{1}^{\prime}(z)+F_{2}^{\prime}(z)}{2}\right)&=\frac{1}{2}\Re\left(p_{2}(z)\right)
+\frac{1}{2}\Re\left(\frac{\frac{z}{(1-z)^{2}}p_{3}(z)*\Psi(z)}{\frac{z}{1-z^{2}}}\right)\\
&=\frac{1}{2}\Re\left(p_{2}(z)\right)+\frac{1}{2}\Re\left(\frac{\frac{z}{(1-z)^{2}}p_{3}(z)*\Psi(z)}{z\Psi^{\prime}(z)}\right)\\
&=\frac{1}{2}\Re\left(p_{2}(z)\right)+\frac{1}{2}\Re\left(\frac{\frac{z}{(1-z)^{2}}p_{3}(z)*\Psi(z)}{\frac{z}{(1-z)^{2}}*\Psi(z)}\right)
>0.
\end{split}
\end{equation*}
By setting $\mu=\nu=\frac{\pi}{2}$ in Lemma \ref{Lem2}, we know that $H+G$ is convex in the direction of imaginary axis. Furthermore, by Lemma \ref{Lem1}, we deduce that $f_{1}*f_{2}$  is univalent and convex in the direction of imaginary axis provided it is locally univalent.
\end{proof}

\vskip.20in
\section{Several examples}

In this section, we give several examples to illustrate our main results.
\begin{example}\label{e1}
{\rm Let $f_{1}=h_{1}+\overline{g_{1}},$ where $h_{1}+g_{1}=\frac{1}{2}\frac{z}{1-z}+\frac{1}{4i}\log\frac{1+iz}{1-iz}$ and $\omega_{1}=z$. Then
\begin{equation*}
h_{1}=\frac{1}{8i}\log\frac{1+iz}{1-iz}+\frac{3}{8}\log(1+z)-\frac{1}{8}\log(1-z)-\frac{1}{8}\log(1+z^{2})+\frac{1}{4}\frac{z}{1-z},
\end{equation*}
 and
\begin{equation*}
g_{1}=\frac{1}{8i}\log\frac{1+iz}{1-iz}-\frac{3}{8}\log(1+z)+\frac{1}{8}\log(1-z)+\frac{1}{8}\log(1+z^{2})+\frac{1}{4}\frac{z}{1-z}.
\end{equation*}
Suppose  that $f_{2}=h_{2}+\overline{g_{2}},$ where $h_{2}+g_{2}=\frac{1}{2}\frac{z}{1-z}+\frac{1}{4i}\log\frac{1+iz}{1-iz}$ and $\omega_{2}=-z^{2}$. Then we have
\begin{equation*}
h_{2}=\frac{1}{8i}\log\frac{1+iz}{1-iz}+\frac{3}{16}\log\frac{1+z}{1-z}+\frac{1}{8}\frac{z}{1-z}+\frac{1}{8}\frac{2z-z^2}{(1-z)^2},
\end{equation*}
 and
 \begin{equation*}
 g_{2}=\frac{1}{8i}\log\frac{1+iz}{1-iz}-\frac{3}{16}\log\frac{1+z}{1-z}+\frac{3}{8}\frac{z}{1-z}-\frac{1}{8}\frac{2z-z^2}{(1-z)^2}.
 \end{equation*}
By Corollary \ref{c1}, we know that $f_{3}=\frac{1}{2}f_{1}+\frac{1}{2}f_{2}$ with $t=\frac{1}{2}$ is convex in the direction of imaginary axis. The images of $\D$ under $f_{j}\ (j=1,2,3)$ are shown in Figures 1-3, respectively.}
\begin{figure}[htbp]
    \centering
\begin{minipage}{0.45\linewidth}
    \centering
    \includegraphics[width=2.0in]{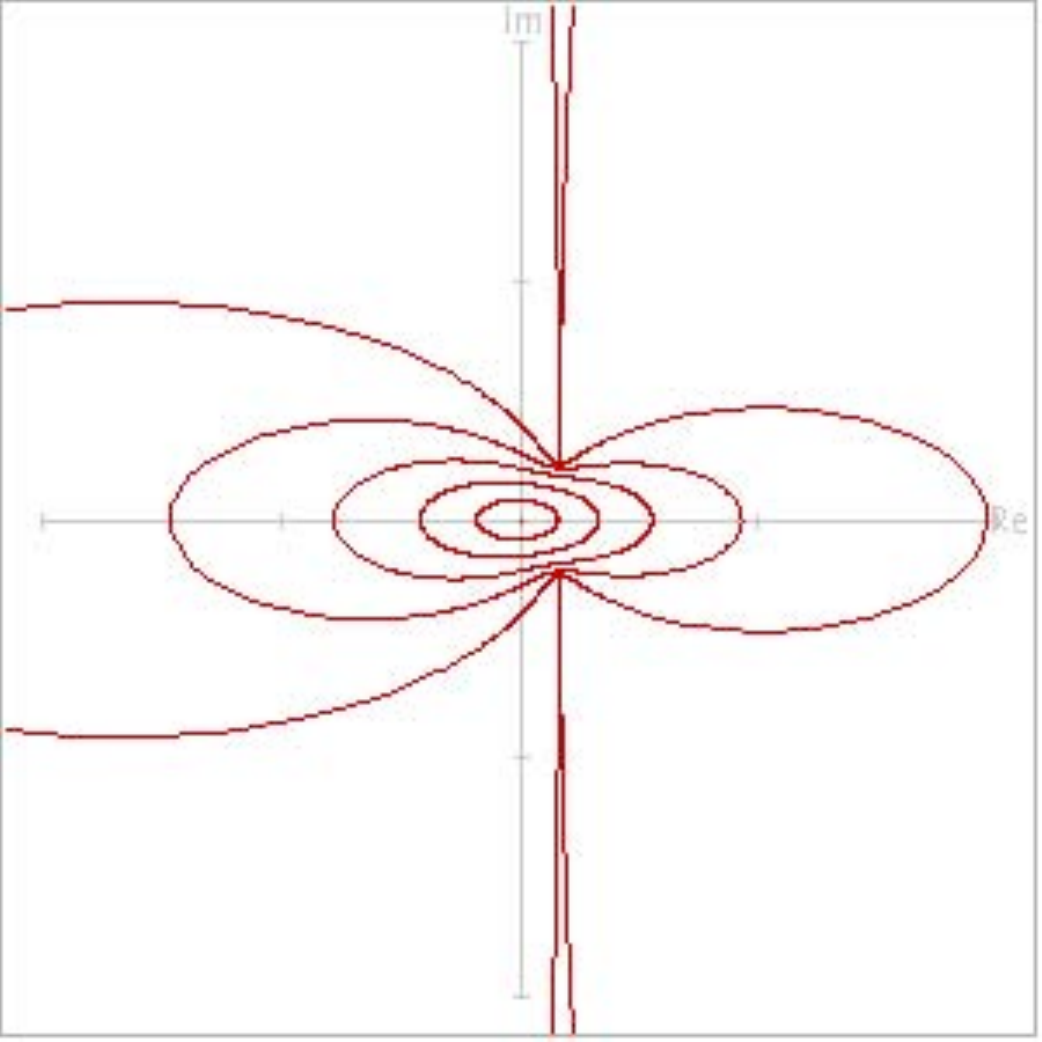}
    \caption{Image of  $f_{1}$ }
\end{minipage}
\hspace{4ex}
\begin{minipage}{0.45\linewidth}
    \centering
    \includegraphics[width=2.0in]{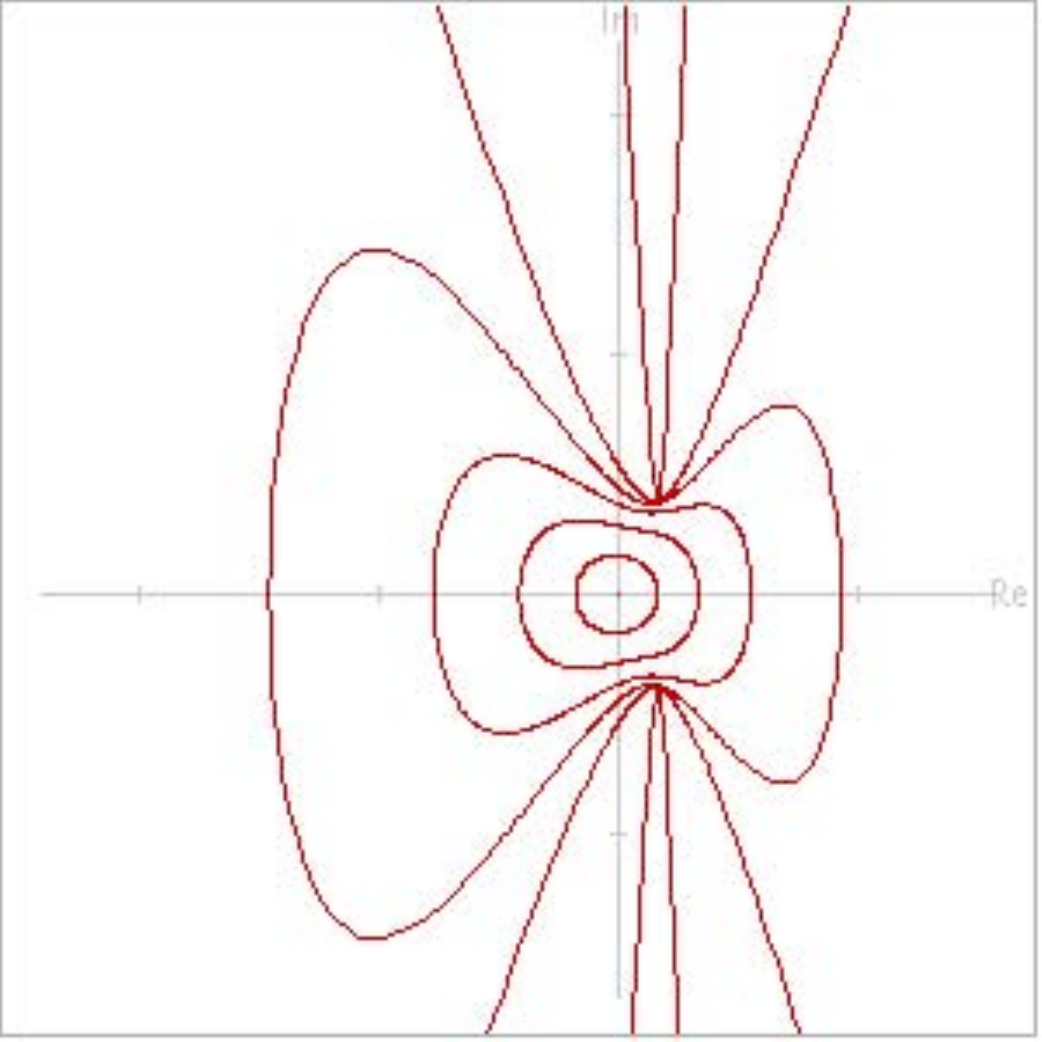}
    \caption{Image of  $f_{2}$ }
\end{minipage}
\end{figure}
\begin{figure}[htbp]
\centering
\includegraphics[width=2.0in]{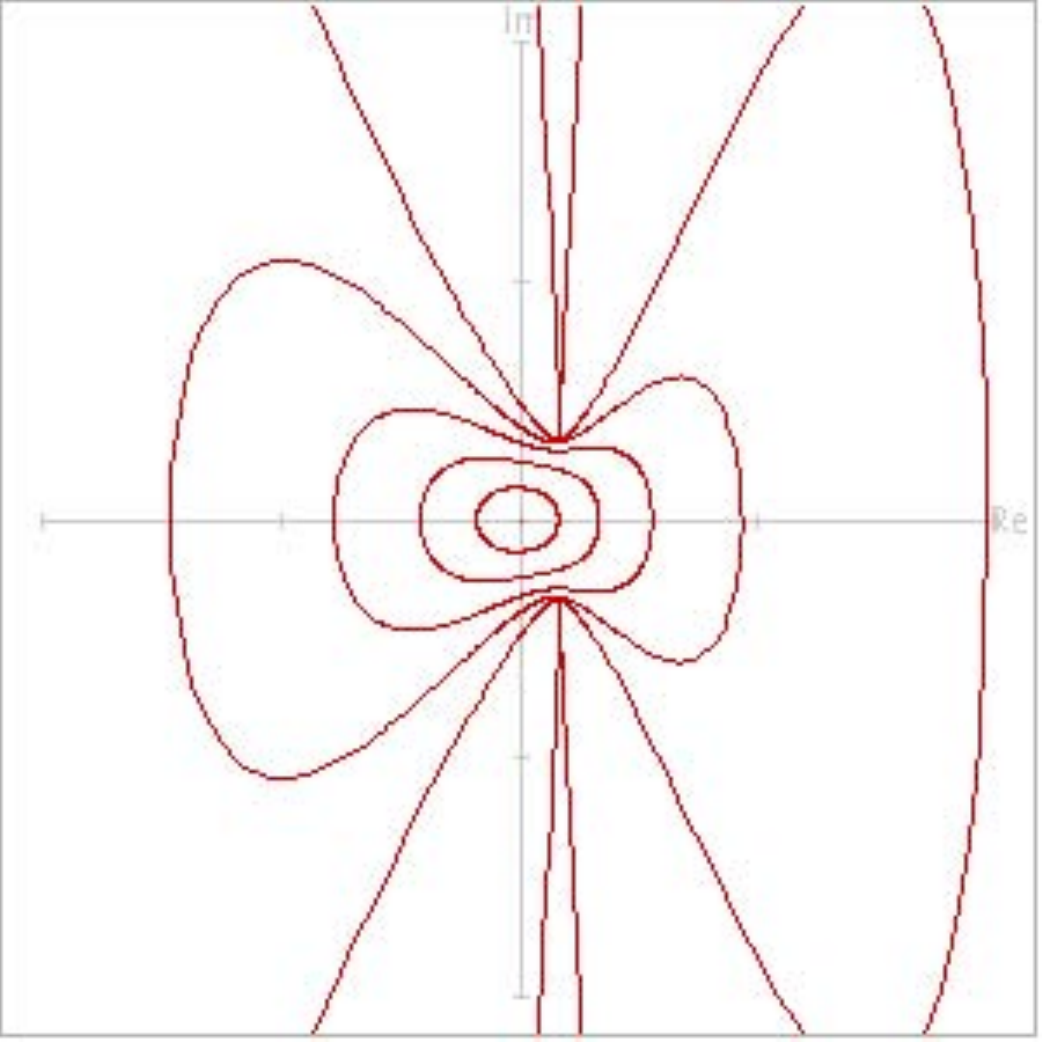}
\caption{Image of  $f_{3}=\frac{1}{2}f_{1}+\frac{1}{2}f_{2}$ }
\end{figure}
\end{example}

\begin{example}\label{e2}
{\rm Let $f_{4}=h_{4}+\overline{g_{4}},$ where $h_{4}-g_{4}=\frac{1}{4}\log\frac{1+z}{1-z}+\frac{1}{2}\frac{z}{(1-z)^2}$ and $\omega_{4}=\frac{1}{2}$. Then we obtain
 \begin{equation*}
 h_{4}=\frac{1}{2}\log\frac{1+z}{1-z}-\frac{1}{1-z}+\frac{1}{(1-z)^2},
\end{equation*}
and
 \begin{equation*}
g_{4}=\frac{1}{4}\log\frac{1+z}{1-z}-\frac{1}{1-z}+\frac{1}{2}\frac{2-z}{(1-z)^2}.
\end{equation*}
Suppose that $f_{5}=h_{5}+\overline{g_{5}},$ where $h_{5}-g_{5}=\frac{1}{4}\log\frac{1+z}{1-z}+\frac{1}{2}\frac{z}{(1-z)^2}$ and $\omega_{5}=-z$. We know that
 \begin{equation*}
 h_{5}=\frac{1}{8}\log\frac{1+z}{1-z}+\frac{1}{4}\frac{z}{1+z}+\frac{1}{4}\frac{1}{(1-z)^2}-\frac{1}{4},
\end{equation*} and
 \begin{equation*}
 g_{5}=-\frac{1}{8}\log\frac{1+z}{1-z}+\frac{1}{4}\frac{z}{1+z}+\frac{1}{4}\frac{1-2z}{(1-z)^2}-\frac{1}{4}.
\end{equation*}
By Corollary \ref{c2}, we know that $f_{6}=\frac{1}{3}f_{1}+\frac{2}{3}f_{2}$ is convex in the direction of real axis. The images of $\D$ under $f_{j}\ (j=4,5,6)$ with $t=\frac{1}{3}$ are presented in Figures 4-6, respectively.}
\begin{figure}[htbp]
    \centering
\begin{minipage}{0.45\linewidth}
    \centering
    \includegraphics[width=2.0in]{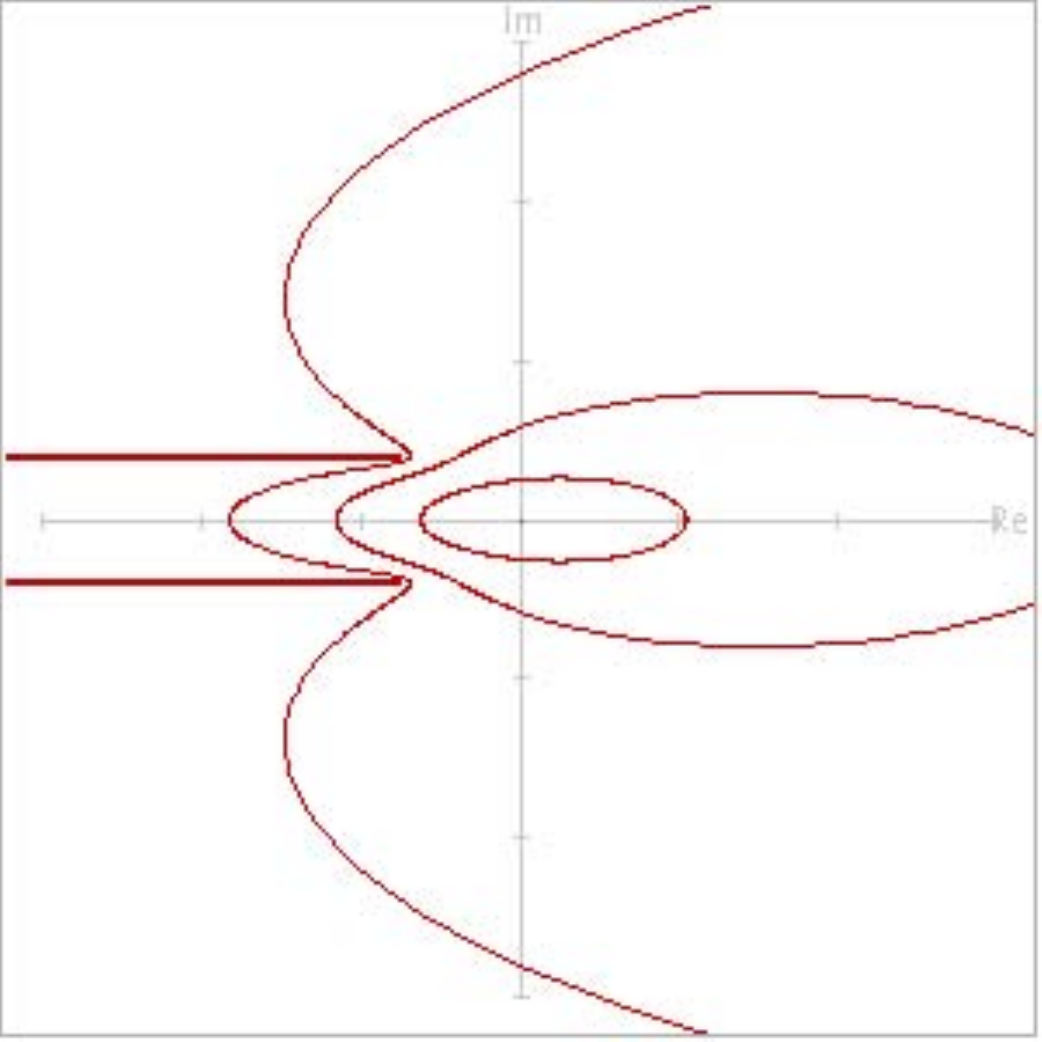}
    \caption{Image of  $f_{4}$ }
\end{minipage}
\hspace{4ex}
\begin{minipage}{0.45\linewidth}
    \centering
    \includegraphics[width=2.0in]{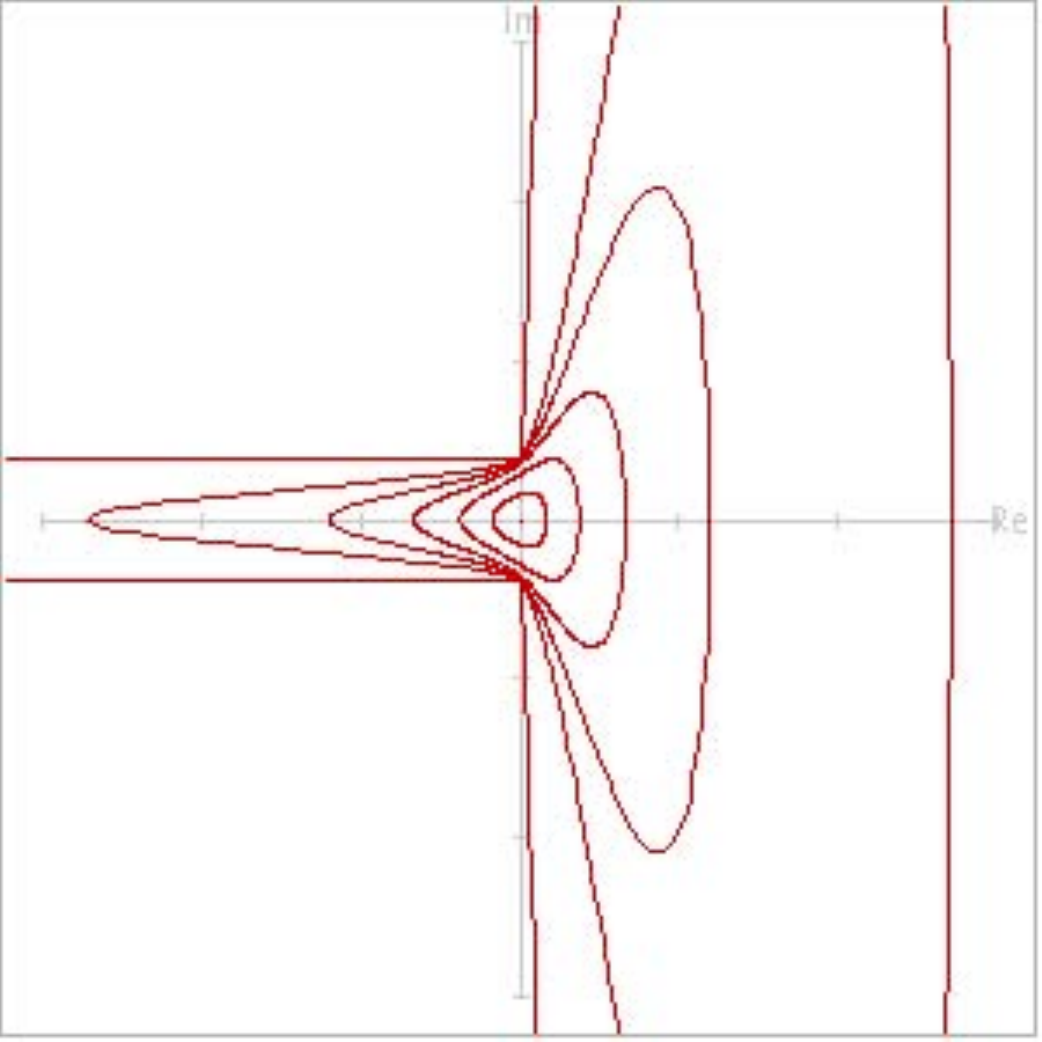}
    \caption{Image of  $f_{5}$ }
\end{minipage}
\end{figure}
\begin{figure}[htbp]
\centering
\includegraphics[width=2.0in]{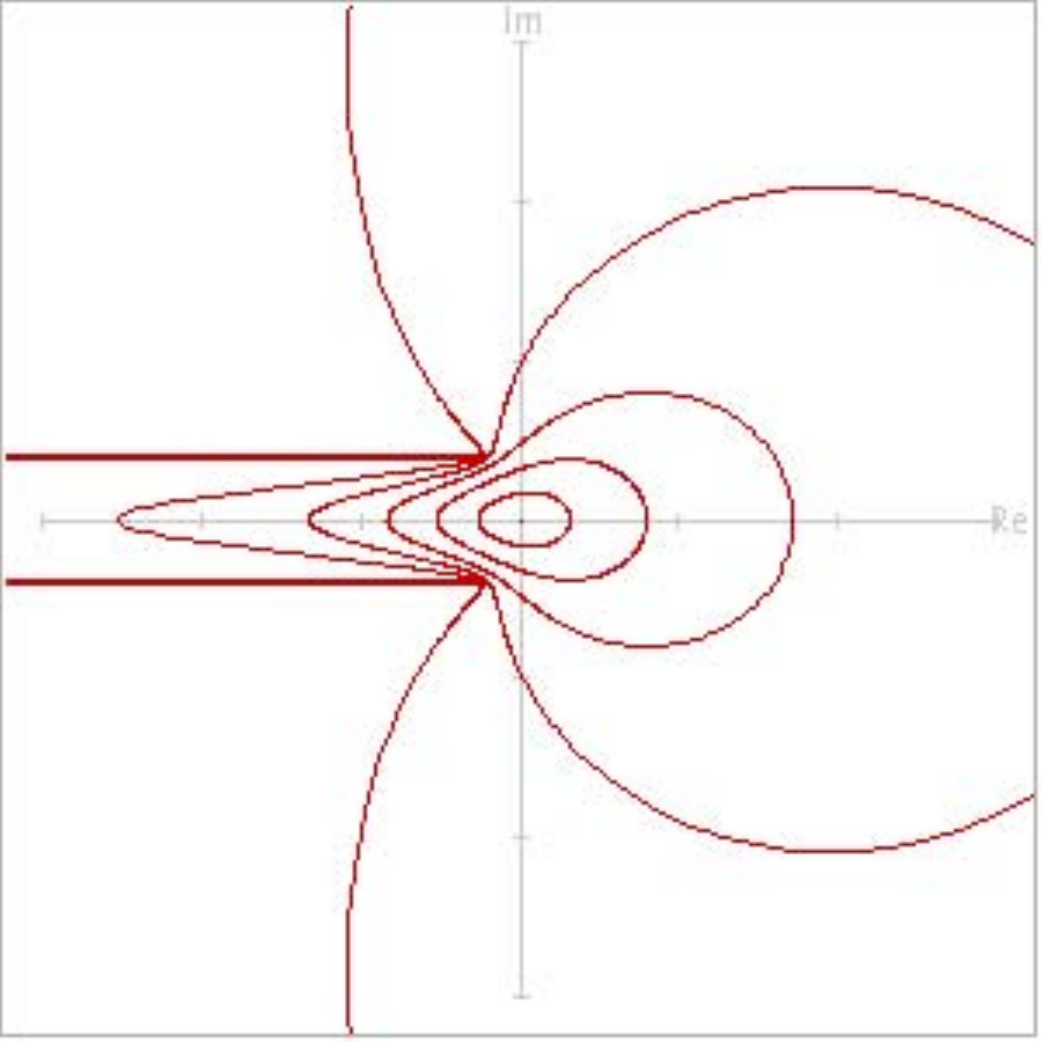}
\caption{Image of  $f_{6}=\frac{1}{3}f_{4}+\frac{2}{3}f_{5}$ }
\end{figure}
\end{example}

\begin{example}\label{e3}
{\rm Taking $\theta_{1}=\frac{\pi}{6}$ and $\theta_{2}=-\frac{\pi}{6}$ in Theorem \ref{t6}, we know that
\begin{equation*}\phi=2\arctan\left(2z+\sqrt{3}\right)-\frac{2\pi}{3}.
\end{equation*}
Let $f_{7}=h_{7}+\overline{g_{7}},$ where $h_{7}+g_{7}=\frac{z}{1-z}$ and $\omega_{7}=-z$. Then we get
 \begin{equation*}
 h_{7}=\frac{z-\frac{1}{2}z^{2}}{(1-z)^2}\quad{\rm{and}}\quad g_{7}=\frac{-\frac{1}{2}z^{2}}{(1-z)^{2}}.
\end{equation*}
Suppose  that $f_{8}=h_{8}+\overline{g_{8}},$ where $h_{8}+g_{8}=2\arctan\left(2z+\sqrt{3}\right)-\frac{2\pi}{3}$ and $\omega_{8}=z^2$. Then we find that
 \begin{equation*}
 h_{8}=\frac{\sqrt{3}}{6}\log\frac{1+\sqrt{3}z+z^2}{1+z^2}+\arctan\left(2z+\sqrt{3}\right)-\frac{\pi}{3},
\end{equation*}
 and
 \begin{equation*}
 g_{8}=-\frac{\sqrt{3}}{6}\log\frac{1+\sqrt{3}z+z^2}{1+z^2}+\arctan
 \left(2z+\sqrt{3}\right)-\frac{\pi}{3}.
\end{equation*}
Let $f_{9}=f_{7}*f_{8}=H+\overline{G}$, we have
\begin{equation*}\begin{split}
H&=h_{7}*h_{8}=\frac{1}{2}\left(h_{8}+zh_{8}^{\prime}\right)\\
&=\frac{1}{2}\left(\frac{z}{\left(1+z^{2}\right)\left(1+\sqrt{3}z+z^2\right)}+\frac{\sqrt{3}}{6}\log\frac{1+\sqrt{3}z+z^2}{1+z^2}
+\arctan\left(2z+\sqrt{3}\right)-\frac{\pi}{3}\right),
\end{split}
\end{equation*}
and
\begin{equation*}\begin{split}
G&=g_{7}*g_{8}=\frac{1}{2}\left(g_{8}-zg_{8}^{\prime}\right)\\
&=\frac{1}{2}\left(-\frac{z^3}{\left(1+z^{2}\right)\left(1+\sqrt{3}z+z^2\right)}-\frac{\sqrt{3}}{6}\log\frac{1+\sqrt{3}z+z^2}{1+z^2}+\arctan\left(2z+\sqrt{3}\right)
-\frac{\pi}{3}\right).
\end{split}
\end{equation*}
A simple calculation shows that the dilatation of $f_{9}=f_{7}*f_{8}$ is given by
$\omega_{9}=-z^2$. Thus, $f_{9}$ is locally univalent.
By Theorem \ref{t6}, we know that
\begin{equation*}\begin{split}
f_{9}&=\Re(H+G)+i\Im(H-G)\\
&=\Re\left(\frac{1}{2}\frac{z\left(1-z^2\right)}{\left(1+z^{2}\right)\left(1+\sqrt{3}z+z^2\right)}+\arctan\left(2z+\sqrt{3}\right)-\frac{\pi}{3}\right)\\
&\quad\quad+i\Im\left(\frac{1}{2}\frac{z}{1+\sqrt{3}z+z^2}+\frac{\sqrt{3}}{6}\log\frac{1+\sqrt{3}z+z^2}{1+z^2}\right)
\end{split}
\end{equation*}
is univalent and convex in the direction of real axis. The images of $\D$ under $f_{j}\ (j=7,8,9)$ are shown in Figures 7-9, respectively.}

\begin{figure}[htbp]
    \centering
\begin{minipage}{0.45\linewidth}
    \centering
    \includegraphics[width=2.0in]{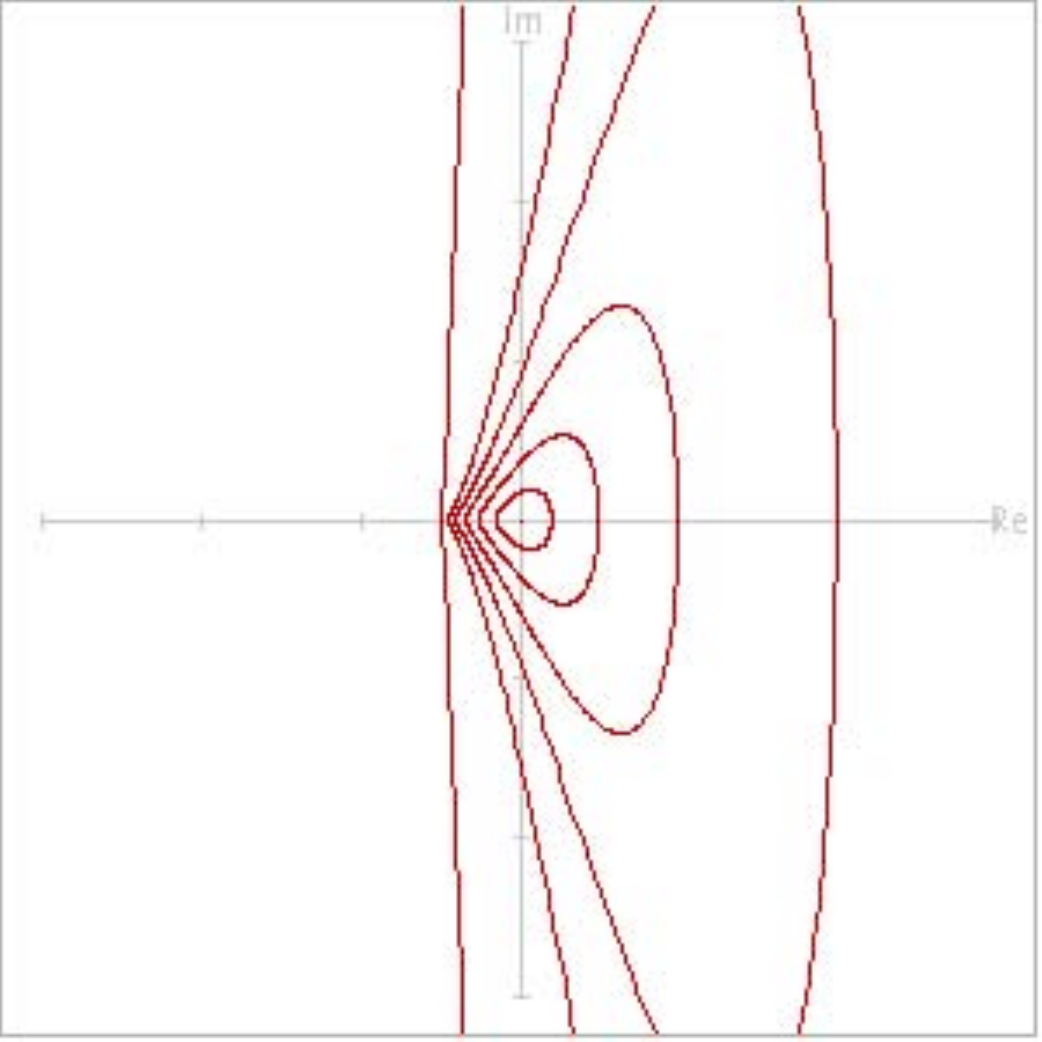}
    \caption{Image of  $f_{7}$ }
\end{minipage}
\hspace{4ex}
\begin{minipage}{0.45\linewidth}
    \centering
    \includegraphics[width=2.0in]{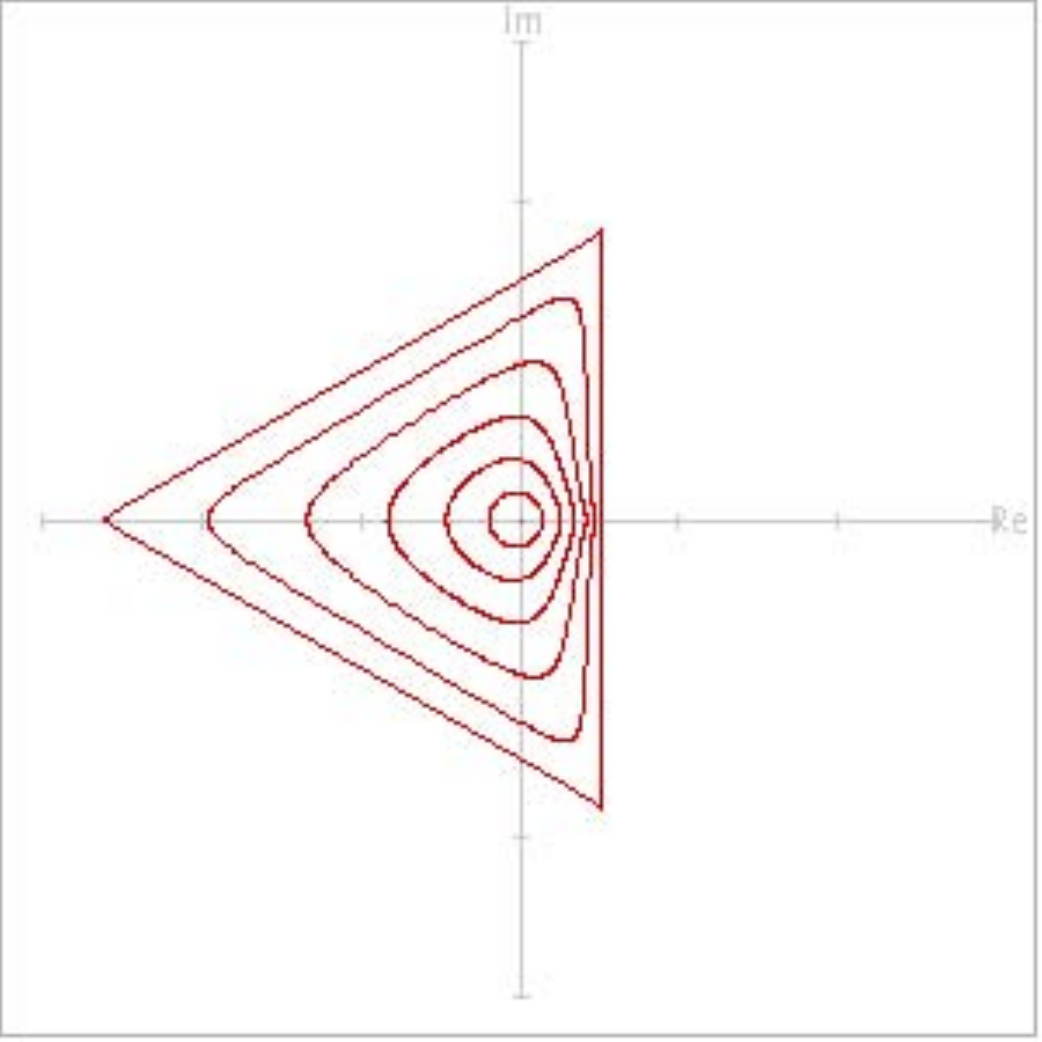}
    \caption{Image of  $f_{8}$ }
\end{minipage}
\end{figure}
\begin{figure}[htbp]
\centering
\includegraphics[width=2.0in]{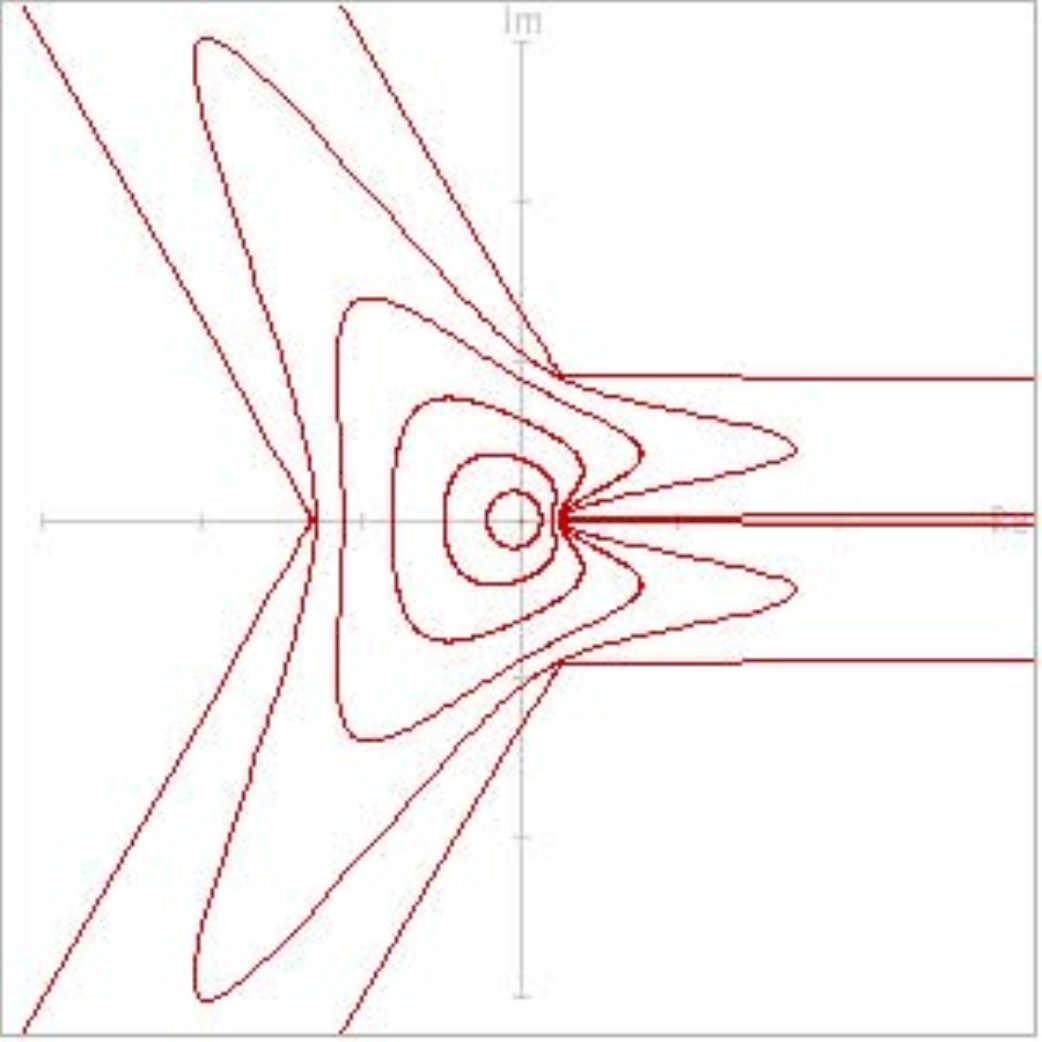}
\caption{Image of  $f_{9}=f_{7}*f_{8}$ }
\end{figure}
\end{example}

\vskip.20in

\begin{center}{\sc Acknowledgments}
\end{center}
\vskip.05in

The present investigation was supported by the \textit{National
Natural Science Foundation} under Grants Nos. 11301008, 11226088 and 11071059 of the P. R. China, and the \textit{Key Project of Natural Science Foundation of
Educational Committee of Henan Province} under Grant No. 12A110002 of
the P. R. China.

\end{document}